\documentclass[12pt, a4paper]{amsart}
\usepackage{amsmath}
\usepackage{amsthm,graphics,tabularx,amssymb,shapepar}
\usepackage{mathtools}
\usepackage{stmaryrd}
\usepackage{fullpage}
\usepackage{enumerate,enumitem}
\usepackage{caption}
\usepackage{tikz-cd}
\usepackage{amscd}
\usepackage[cmyk]{xcolor}
\usepackage{comment}
\excludecomment{review}

\usepackage{hyperref}
\usepackage[
backend=biber,
style=alphabetic,
sorting=nyt,
maxnames=9,
minnames=1,
maxalphanames=99,
]{biblatex}
\makeatletter
\newcommand*{\rom}[1]{\expandafter\@slowromancap\romannumeral #1@}
\makeatother

\newcommand{\BC}{{\mathbb {C}}}

\newcommand{\BR}{{\mathbb {R}}}

\newcommand{\BZ}{{\mathbb {Z}}}

\newcommand{\CH}{{\mathcal {H}}}

\newcommand{\CN}{{\mathcal {N}}}

\newcommand{\CP}{{\mathcal {P}}}

\newcommand{\CR}{{\mathcal {R}}}

\newcommand{\CY}{{\mathcal {Y}}}
\newcommand{\CZ}{{\mathcal {Z}}}

\newcommand{\RO}{{\mathrm {O}}}

\newcommand{\RU}{{\mathrm {U}}}

\newcommand{\End}{{\mathrm{End}}}

\newcommand{\Hom}{{\mathrm{Hom}}}

\newcommand{\Ind}{{\mathrm{Ind}}}

\newcommand{\Sp}{{\mathrm{Sp}}}

\newcommand{\Span}{{\mathrm{Span}}}

\newcommand{\degree}{\operatorname{deg}}

\newcommand{\sgn}{\operatorname{sgn}}

\newcommand{\LD}{\operatorname{LD}}

\newcommand{\g}{\mathfrak g}

\renewcommand{\u}{\mathfrak u}

\newcommand{\m}{\mathfrak m}

\renewcommand{\sp}{\mathfrak s \mathfrak p}

\newcommand{\be}{\begin {equation}}
\newcommand{\ee}{\end {equation}}
\newcommand{\bee}{\begin {equation*}}
\newcommand{\eee}{\end {equation*}}

\newcommand{\JH}{{Jordan--H\"{o}lder }}

\DeclareMathOperator{\otimeshat}{\widehat{\otimes}}

\newcommand{\wtd}[1]{\widetilde{#1}}

\newcommand{\tG}{\widetilde G}
\newcommand{\tGp}{\widetilde G'}
\newcommand{\tK}{\widetilde K}
\newcommand{\tKp}{\widetilde K'}
\newcommand{\tU}{\widetilde U}

\newtheorem{thm}{Theorem}[section]
\newtheorem{cort}[thm]{Corollary}
\newtheorem{lemt}[thm]{Lemma}

\newtheorem{dfnt}[thm]{Definition}

\theoremstyle{remark}

\theoremstyle{remark}
\newtheorem*{rremark}{Remark}

\begin{document}
\title{Multiplicity One Property for Archimedean Theta Correspondence}

\author[Z. Geng]{Zhibin Geng}
\address{Department of Mathematics and New Cornerstone Science Laboratory, The University of Hong Kong}
\email{gengzb@hku.hk}

\author[K. Wu]{Kaidi Wu}
\address{Department of Mathematics and New Cornerstone Science Laboratory, The University of Hong Kong}
\email{kaidiwu24@connect.hku.hk}

\date{\today}

\keywords{Theta lifting, \JH multiplicity, lowest degree $K$-type}
\subjclass[2020]{11F27, 22E50}

\begin{abstract}
    For any archimedean reductive dual pair, we prove that the \JH multiplicity of the small theta lift in the full theta lift is one.
    The main input is an equality between the multiplicities of the lowest degree $K$-types in the full and small theta lifts.
    We also deduce a new description of the archimedean theta correspondence in terms of lowest degree $K$-types.
\end{abstract}

\maketitle
\tableofcontents

\section{Introduction}
Let $\Sp_{2n}(\BR) = \Sp_{2n} = \Sp$ be the real symplectic group of rank $n$, and let $(G, G')$ be a reductive dual pair in $\Sp$.
Fix a maximal compact subgroup $U\simeq U(n)$ of $\Sp$
such that
\[
    K:=G\cap U,\qquad K':=G'\cap U
\]
are maximal compact subgroups of $G$ and $G'$, respectively.
We use lowercase Gothic letters to denote the complexified
Lie algebras of the corresponding groups, e.g. $\sp, \u, \g, \g'$.
Let $\wtd{\Sp}$ be the metaplectic double cover of $\Sp$. For any subgroup $E\subseteq\Sp$, denote by $\wtd{E}$ the inverse image of $E$ in $\wtd{\Sp}$.
Note that $\tG$ commutes with $\tGp$. We use $\tG\cdot\tGp$ to denote their product subgroup in $\wtd{\Sp}$, and use $\tK\cdot\tKp$ similarly.

Fix a nontrivial unitary character $\psi$ of $\BR$, and let
$\omega^\infty$ be the associated smooth Weil representation
of $\wtd{\Sp}$.
Let $\omega$ be its $(\sp,\tU)$-module of $\tU$-finite vectors.
We realize $\omega$ in the Fock model
\[
    \CP=\BC[z_1,\ldots,z_n].
\]
By restriction, $\CP$ is a
$(\g\oplus\g',\tK\cdot\tKp)$-module.

Denote by $\CR(\g,\tK,\omega)$ the set of isomorphism classes of irreducible $(\g,\tK)$-modules $\pi$ such that
\[
    \Hom_{\g,\tK}(\CP,\pi)\neq 0.
\]
Define $\CR(\g',\tKp,\omega)$ similarly.
For $\pi \in \CR(\g,\tK,\omega)$, consider 
\begin{equation}\label{Eq: N pi}
    \CN_{\pi} := \bigcap_{\phi \in \Hom_{\g, \tK}(\CP, \pi)} \ker \phi,
\end{equation}
which is stable under $(\g \oplus \g', \tK \cdot \tKp)$-action.
The following celebrated theorem is due to Howe.

\begin{thm}[{\cite[Theorem 2.1]{Howe1989TranscendingClassicalInvTheory}}]
  For every $\pi\in\CR(\g,\tK,\omega)$, there is an isomorphism
  \[
    \CP/\CN_\pi\simeq\pi\otimes\Theta(\pi)
  \]
  of $(\g\oplus\g',\tK\cdot\tKp)$-modules, where $\Theta(\pi)$ is a $(\g',\tKp)$-module of finite length.
  Moreover, $\Theta(\pi)$ has a unique irreducible quotient $\theta(\pi)$, and the correspondence $\pi\mapsto\theta(\pi)$ defines a bijection from $\CR(\g,\tK,\omega)$ to $\CR(\g',\tKp,\omega)$.
\end{thm}

We call $\Theta(\pi)$ and $\theta(\pi)$ the full and small
theta lifts of $\pi$, respectively. 
Although small theta lifts have been explicitly described in many cases, the structure of full theta lifts remains less well understood. A basic question is to determine the \JH multiplicity of $\theta(\pi)$ in $\Theta(\pi)$, i.e., the multiplicity of $\theta(\pi)$ occurs in the \JH factors of $\Theta(\pi)$, which we denote by $[\Theta(\pi):\theta(\pi)]$. 
Our main result is the following.

\begin{thm}\label{Thm: intro-JH}
  For any $\pi$ in $\CR(\g,\tK,\omega)$, we have 
  \[
  [\Theta(\pi) : \theta(\pi)] = 1.
  \]
\end{thm}

This theorem is useful for studying the structure of the full theta lift $\Theta(\pi)$ and proving certain irreducibility results. We will present applications in this direction in a forthcoming paper.

Theorem~\ref{Thm: intro-JH} also admits an analogue for smooth theta lifts, which can be deduced from this algebraic result. We will discuss this in Subsection~\ref{Subsec: smooth vertion}.

The main input of the proof of Theorem~\ref{Thm: intro-JH} is indeed an analysis of lowest degree $\tK$-types. We have the following equality between the multiplicities of the lowest degree $\tK$-types in the full and small theta lifts.

\begin{thm}\label{Thm: intro-LDK}
  For any $\pi$ in $\CR(\g,\tK,\omega)$, 
  \begin{enumerate}
      \item The sets of lowest degree $\tKp$-types in $\Theta(\pi)$ and $\theta(\pi)$ are the same.
      \item For any lowest degree $\tKp$-type $\tau$ in $\Theta(\pi)$ (thus also of lowest degree in $\theta(\pi)$), we have
        \[
        \dim \Hom_{\tKp}(\tau,\Theta(\pi))
        = \dim \Hom_{\tKp}(\tau,\theta(\pi)).
        \]
  \end{enumerate}
\end{thm}

The proof of Theorem~\ref{Thm: intro-LDK} is based on the observation that a simultaneous cyclicity result in \cite{Howe1989TranscendingClassicalInvTheory} can be used to compare the multiplicities of corresponding lowest degree $\tK$-types and $\tKp$-types. 
We compare these multiplicities in both theta lifts with the multiplicity of the corresponding $\tK$-type in $\pi$ and obtain
the desired equality.

A natural further question is whether these lowest degree $\tKp$-types always occur with multiplicity one. In general, the answer is negative. In Section~\ref{Sec: mult two example}, we give an example in which a lowest degree $\tKp$-type occurs with multiplicity two.

Using the above analysis of lowest degree $\tKp$-types, we have a new description of the theta correspondence.

\begin{cort}\label{Cor:intro-characterization}
The small theta lift $\theta(\pi)$ is the unique irreducible \JH constituent of $\Theta(\pi)$ that contains a (and hence every) lowest degree $\tKp$-type of $\Theta(\pi)$.  
%More strongly, for every $\sigma'\in\LD_{K'}(\Theta(\pi))$, the unique irreducible constituent of $\Theta(\pi)$ containing $\sigma'$ is $\theta(\pi)$.\ZG{the second statement is really stronger than the first one? I'm not sure}
\end{cort}

The paper is organized as follows.
In Section~\ref{Sec: preliminary}, we recall basic notions in the theory of theta correspondence and results from Howe's original proof.
In Section~\ref{Sec: LDK}, we compare the lowest degree $\tKp$-types in the full and small theta lifts, and prove Theorem~\ref{Thm: intro-LDK}.
In Section~\ref{Sec: Proof main}, we prove Theorem~\ref{Thm: intro-JH} and its smooth analog. 
Finally, in Section~\ref{Sec: mult two example}, we give an example in which a lowest degree $\tKp$-type occurs with multiplicity two.

\section{Preliminaries}\label{Sec: preliminary}
In this section, we recall some preliminary material which is mainly contained in Howe's original paper \cite{Howe1989TranscendingClassicalInvTheory} and Adams' survey note \cite{Adams2007ThetaR}.
We also fix the notation to be used in the paper.

We retain the notation of the introduction.
Recall that $\CP=\BC[z_1,\ldots,z_n]$.
Let
\[
    \CP:=\bigoplus_{d\geq0}\CP_d,\qquad
    \CP^{(d)}:=\bigoplus_{0\leq j\leq d}\CP_j
    \quad(d\geq0),\qquad
    \CP^{(-1)}:=0,
\]
where $\CP_d$ consists of the homogeneous polynomials of degree $d$.
Through the action of $\omega$, we identify $\sp$ with its
image in the algebra of differential operators on $\CP$.
It has decomposition
\[
 \sp=\sp^{(1,1)}\oplus\sp^{(2,0)}\oplus\sp^{(0,2)},
\]
where,
\begin{equation*}
 \begin{aligned}
 \sp^{(1,1)}&=\Span_{\BC}\{z_i\frac{\partial}{\partial z_j} + \frac{\partial}{\partial z_j} z_i \mid 1\leq i,j\leq n\},\\
 \sp^{(2,0)}&=\Span_{\BC}\{z_i z_j \mid 1\leq i,j\leq n\},\\
 \sp^{(0,2)}&=\Span_{\BC}\{\frac{\partial^2}{\partial z_i\partial z_j} \mid 1\leq i,j\leq n\}.
 \end{aligned}
\end{equation*}
Under this identification, we have $\u=\sp^{(1,1)}$. Note that these three spaces consist of operators preserving degree, raising degree by two, and lowering degree by two, respectively; see \cite[(2.2)--(2.3)]{Howe1989TranscendingClassicalInvTheory}. Thus, each $\CP_d$ is a finite-dimensional $\widetilde U$-module.

We now introduce the notion of degree for a $\tK$-type, which plays an important role in the theory of archimedean theta correspondence.

\begin{dfnt}
Let $\sigma$ be an irreducible representation of $\tK$. Define
\[
    \deg\sigma:=
    \min\{d\geq0 \mid \Hom_{\tK}(\sigma,\CP_d)\neq0\}.
\]
We adopt the convention that $\deg \sigma := \infty$ if $\Hom_{\tK}(\sigma,\CP) = 0$. 
Similarly, one can define the degree of an irreducible representation of $\tKp$.
\end{dfnt}

\begin{rremark}
    The notion of degree depends on the fixed Weil representation.
\end{rremark}

We will use the following notation for representations of compact groups.
For a compact group $L$, let $\CR(L)$ be the set of equivalence classes of irreducible $L$-representations.
For a locally finite $L$-representation $V$, denote by $V_\sigma$ the $\sigma$-isotypic component of $V$, and set
\[
 \CR(L, V):=\{\sigma\in\CR(L) \mid \Hom_L(\sigma,V)\ne0\}.
\]
\begin{dfnt}
Let $L$ be either $\tK$ or $\tKp$, and let $V$ be a nonzero locally finite $L$-representation such that $\CR(L,V)\subseteq\CR(L,\CP)$. Define
\[
    \LD(L,V):= \{ \sigma\in\CR(L,V) \mid  \deg\sigma= \min_{\tau\in\CR(L,V)}\deg\tau \}.
\]
The elements of $\LD(L,V)$ are called the lowest degree
$L$-types of $V$.
\end{dfnt}
Since the degrees of the $L$-types of $V$ are finite nonnegative integers, the set $\LD(L,V)$ is nonempty.
In particular, for $\pi\in\CR(\g,\tK,\omega)$, the set $\LD(\tK,\pi)$ is well-defined and nonempty.

Recall that there are auxiliary reductive dual pairs $(K,M')$ and $(M,K')$ in $\Sp_{2n}(\BR)$, with inclusions
\begin{equation*}\label{Eq: dual pair}
\begin{array}{ccccc}
M  & \longleftrightarrow & K' \\[-2pt]
\rotatebox{90}{$\subseteq$} && \rotatebox{90}{$\supseteq$} \\[-2pt]
G  & \longleftrightarrow & G' \\[-2pt]
\rotatebox{90}{$\subseteq$} && \rotatebox{90}{$\supseteq$} \\[-2pt]
K  & \longleftrightarrow & M'
\end{array},
\end{equation*}
and each row forms a reductive dual pair. 
For $(a,b)=(1,1),(2,0),(0,2)$, let
\[
 \m^{(a,b)}:=\m\cap\sp^{(a,b)},\qquad
 \m'^{(a,b)}:=\m'\cap\sp^{(a,b)}.
\]
Since $K,K'\subseteq U$, $\m$ and $\m'$ also admit decomposition as $\sp$,
\[
 \begin{aligned}
 \m&=\m^{(1,1)}\oplus\m^{(2,0)}\oplus\m^{(0,2)},\\
 \m'&=\m'^{(1,1)}\oplus\m'^{(2,0)}\oplus\m'^{(0,2)}.
 \end{aligned}
\]

\begin{dfnt}[{\cite[(3.6)]{Howe1989TranscendingClassicalInvTheory}}]
\label{Def: harmonics}
The spaces of $K$-harmonic and $K'$-harmonic polynomials are
\[
 \begin{aligned}
 \CH(K)&:=\{f\in\CP \mid Xf=0\ \text{for every }X\in\m'^{(0,2)}\},\\
 \CH(K')&:=\{f\in\CP \mid Xf=0\ \text{for every }X\in\m^{(0,2)}\}.
 \end{aligned}
\]
The joint harmonics space $\CH$ is 
\[
\CH:=\CH(K)\cap\CH(K').
\]
All three spaces are stable under $\widetilde K\cdot\widetilde K'$ action.

For $\sigma \in \CR(\tK)$ and $\sigma' \in \CR(\tKp)$, set
\[
\CH_{\sigma, \sigma'} := \CH(K)_{\sigma}\cap\CH(K')_{\sigma'},
\]
which is the $(\sigma,\sigma')$-isotypic component of $\CH$.
\end{dfnt}

The following result describes the correspondence in joint harmonics and its compatibility with lowest degree types under theta lifting.

\begin{lemt}[{\cite[Lemma 3.3]{Howe1989TranscendingClassicalInvTheory}, \cite[Theorem 7.3]{Adams2007ThetaR}}]\label{Lem: joint harmonic correspondence}
\begin{enumerate}
\item There is a bijection
\[
\theta_{\CH}: \CR(\tK,\CH)\longleftrightarrow\CR(\tKp,\CH).
\]
characterized by
\[
\CH_{\sigma,\sigma'} \neq 0 \quad\Longleftrightarrow\quad \sigma'=\theta_{\CH}(\sigma).
\]
Moreover, when $\sigma'=\theta_{\CH}(\sigma)$, we have
\[
\CH_{\sigma,\sigma'}\simeq\sigma\otimes\sigma',\qquad
 \deg \sigma=\deg \sigma'.
\]

\item Suppose $\pi\in\CR(\g,\tK,\omega)$ and $\sigma$ is a lowest degree $\tK$-type of $\pi$. Then $\sigma\in\CR(\tK,\CH)$.
Moreover, $\theta_{\CH}(\sigma)$ is a lowest degree $\tK'$-type of $\theta(\pi)$.
\end{enumerate}
\end{lemt}

We next recall the description of lowest degree $\tK$-types in nonzero quotients of $\CP$.
Let $\CN \subsetneq \CP$ be a proper $(\g \oplus \g', \tK \cdot \tKp)$-submodule, and consider the quotient $\CP/\CN$.
Denote by $\CY^{(d)}$ the $(\g \oplus \g', \tK \cdot \tKp)$-module generated by $\CP^{(d)}$. Let $\delta(\CN)$ be the largest integer $d$ such that $\CY^{(d)}$ is contained in $\CN$.

\begin{lemt}[{\cite[Lemma 4.1]{Howe1989TranscendingClassicalInvTheory}}]\label{Lem: Howe quotient degree}
\begin{enumerate}[label=\textup{(\alph*)}]
    \item
    %The number $\delta(\CN)$ is determined solely by the structure of $\CP/\CN$ as a $\tK$-module, or as a $\tKp$-module.
    We have
    \begin{equation*}
        \begin{aligned}
            1+\delta(\CN)
            &=
            \min\bigl\{
                \degree \sigma :
                \sigma \in \CR(\tK,\CP/\CN)
            \bigr\} \\
            &=
            \min\bigl\{
                \degree \sigma' :
                \sigma' \in \CR(\tKp,\CP/\CN)
            \bigr\}.
        \end{aligned}
    \end{equation*}

    \item 
    If $\sigma \in \CR(\tK,\CP/\CN)$ and
    $\degree \sigma=\delta(\CN)+1$, then
    \[
        \sigma \in \CR(\tK, \CH) \quad \text{and} \quad \theta_{\CH}(\sigma) \in \CR(\tKp,\CP/\CN).
    \]
\end{enumerate}
\end{lemt}

Together with the equality of degrees in the joint harmonic
correspondence, we obtain the following consequence.

\begin{cort}\label{Cor: LDK correspondence}
    Notation as in Lemma~\ref{Lem: Howe quotient degree}.
    The joint harmonic correspondence restricts to a bijection
    \[
    \theta_{\CH}:\LD(\tK, \CP/\CN)
    \xrightarrow{\ \sim\ }
    \LD(\tKp, \CP/\CN).
    \]
\end{cort}

\begin{proof}
    According to Lemma~\ref{Lem: Howe quotient degree}, we know that the map 
    \[
    \theta_{\CH}:\LD(\tK, \CP/\CN)
    \xrightarrow{}
    \CR(\tKp,\CP/\CN)
    \]
    is well-defined. 
    Note that the statement in Lemma~\ref{Lem: Howe quotient degree} is symmetric, thus we only need to show 
    \[
    \theta_{\CH}\bigl(\LD(\tK,\CP/\CN)\bigr)
    \subseteq\LD(\tKp,\CP/\CN).
    \]
    This follows from $\deg \sigma = \deg \theta_{\CH}(\sigma)$ as in Lemma~\ref{Lem: joint harmonic correspondence}.
\end{proof}

We will also need an intermediate space $\CZ$ introduced in \cite[Equation (4.2)]{Howe1989TranscendingClassicalInvTheory}, where it is denoted by $\CZ^{\tau, \tau'}$. 
Since the explicit construction of $\CZ$ will not be needed here, we only record the properties of $\CZ$ that are relevant to our argument.

\begin{lemt}[{\cite[Equation (4.8)]{Howe1989TranscendingClassicalInvTheory}}]\label{Lem: Z existence}
For $\pi \in \CR(\g,\tK,\omega)$, consider $\CN_{\pi}\subseteq\CP$ as in Equation \eqref{Eq: N pi}.
Take $\sigma \in \LD(\tK, \CP/\CN_{\pi})$, and let $\sigma' := \theta_{\CH}(\sigma)$, $d :=\deg\sigma-1$.
Then there exists a $(\g\oplus \g',\tK\cdot \tKp)$-module
\[
\CZ\subseteq \CP/\CY^{(d)}
\]
with the following properties:
\begin{enumerate}
    \item the natural image of
    $\CH_{\sigma,\sigma'}$ in $\CP/\CY^{(d)}$
    is nonzero and is contained in $\CZ$.

    \item the natural quotient map
    \[
    \CP/\CY^{(d)}\twoheadrightarrow \CP/\CN_{\pi}
    \]
    restricts to a surjection
    \[
    \CZ\twoheadrightarrow\CP/\CN_{\pi}.
    \]

    \item Let
    \[
    M_{\CZ} :=
    \Hom_{\tK\cdot \tKp}(\CH_{\sigma,\sigma'},\CZ),
    \]
    and denote by $e_0\in M_{\CZ}$ the element defined by the natural inclusion $\CH_{\sigma,\sigma'}\hookrightarrow {\CZ}$, then
    \[
    M_{\CZ}
    =
    U(\mathfrak g)^K (e_0)
    =
    U(\mathfrak g')^{K'} (e_0).
    \]
\end{enumerate}
\end{lemt}

\begin{rremark}
The module $\CZ$ (and thus the associated objects $M_{\CZ}$ and $e_0$) depends on the choice of the lowest degree $\tK$ type $\sigma$. We omit this dependence in the notation.
\end{rremark}

\section{Lowest degree $K$-types of theta lifts}\label{Sec: LDK}

In this section, we study the lowest degree $\tKp$-types in full and small theta lifts.
We first show that the sets of lowest degree $\tKp$-types in the full and small theta lifts coincide.
\begin{thm}\label{Thm: LDK coincide}
Let $\pi \in \CR(\g,\tK,\omega)$.  Then
\[
 \LD(\tKp, \Theta(\pi))
 =\LD(\tKp, \theta(\pi))
 =\theta_{\CH}\bigl(\LD(\tK, \pi)\bigr).
\]
\end{thm}

\begin{proof}
Let $V$ be either $\Theta(\pi)$ or $\theta(\pi)$, and let $Q_V=\pi\otimes V$ be the corresponding quotient of $\CP$.  
Since $\tK$ only acts on the first factor and $\tKp$ only acts on the second, we have
\[
 \CR(\tK, Q_V)= \CR(\tK, \pi),
 \qquad
 \CR(\tKp, Q_V)= \CR(\tKp, V).
\]
Therefore,
\[
 \LD(\tK, Q_V)= \LD(\tK, \pi),
 \qquad
 \LD(\tKp, Q_V)= \LD(\tKp, V).
\]

Applying Corollary~\ref{Cor: LDK correspondence} to $Q_V$, we obtain
\[
 \LD(\tKp, V)
 =\theta_{\CH}\bigl(\LD(\tK, \pi)\bigr).
\]
The theorem follows from applying this to both choices of $V$.
\end{proof}

Now we turn to the multiplicities of lowest degree $K$-types in the full and small theta lifts. We will use the following two lemmas.

\begin{lemt}\label{Lem: rank eq}
Let $E$ and $F$ be nonzero finite-dimensional vector spaces, let
$A\subseteq\End(E)$ and $B\subseteq\End(F)$ be subalgebras, and let
$x\in E\otimes F$.  If
\begin{equation*}
 E\otimes F=(A\otimes 1)x=(1\otimes B)x,
\end{equation*}
then $\dim E=\dim F$.
\end{lemt}

\begin{proof}
  By symmetry, it suffices to prove that $\dim F\leq\dim E$.
  Consider the following map associated to $x$,
\[
  T_x:F^*\longrightarrow E,
  \qquad T_x(\lambda)=(1\otimes\lambda)x.
\]
If $\lambda\in\ker T_x$, then for every $a\in A$,
\[
 (1\otimes\lambda)((a\otimes1)x)=aT_x(\lambda)=0.
\]
Since $(A\otimes1)x = E\otimes F$, we have $1\otimes\lambda$ vanishes on all of $E\otimes F$, and hence $\lambda=0$.  
Thus $T_x$ is injective and $\dim F\leq\dim E$.
\end{proof}

\begin{lemt}\label{Lem: theta cyclic}
For $\pi \in \CR(\g,\tK,\omega)$, let $V$ be either $\Theta(\pi)$ or $\theta(\pi)$, and let $Q_V :=\pi \otimes V$.  
Take $\sigma, \CZ, e_0, M_{\CZ}$ as in Lemma~\ref{Lem: Z existence}.
The quotient maps
\[
  \CZ \twoheadrightarrow \CP/\CN_{\pi} \cong \pi \otimes \Theta(\pi)
  \twoheadrightarrow \pi \otimes \theta(\pi)
\]
induce surjections
\[
  q_V: M_{\CZ}
  \twoheadrightarrow
  M_V:=\Hom_{\tK\cdot \tKp}(\CH_{\sigma,\sigma'},Q_V).
\]
Denote by $e_V$ the image of $e_0$ under $q_V$. Then we have $e_V$ is nonzero and
\[
  M_V
  =\mathcal U(\g)^K(e_V)
  =\mathcal U(\g')^{K'}(e_V).
\]
\end{lemt}

\begin{proof}
  Note that $q_V$ is $U(\g)^K \times U(\g')^{K'}$ equivariant. Therefore, according to Lemma~\ref{Lem: Z existence}, we have
  \[
  M_V
  =\mathcal U(\g)^K(e_V)
  =\mathcal U(\g')^{K'}(e_V).
  \]
  Since $M_V$ is non-zero by Lemma~\ref{Lem: joint harmonic correspondence}, we have $e_V \neq 0$.
\end{proof}

\begin{rremark}
    A statement similar to Lemma~\ref{Lem: theta cyclic} already appears in Howe's original proof of the main theorem, see \cite[P547]{Howe1989TranscendingClassicalInvTheory}. We record it in the present form for clarity and later use.
\end{rremark}

We can now show that every lowest degree $\tKp$-type occurs with the same multiplicity in the full and small theta lifts.

\begin{thm}\label{Thm: LDK multiplicity eq}
  For any $\pi$ in $\CR(\g,\tK,\omega)$, let $\sigma$ be a lowest degree $\tK$-type of $\pi$, and $\sigma' := \theta_{\CH}(\sigma)$, then we have
    \[
    \dim \Hom_{\tK}(\sigma,\pi)
    = \dim \Hom_{\tKp}(\sigma',\Theta(\pi))
    = \dim \Hom_{\tKp}(\sigma',\theta(\pi)).
    \]
\end{thm}

\begin{proof}
    Let $V$ be either $\Theta(\pi)$ or $\theta(\pi)$, and denote $Q_V :=\pi \otimes V$.
    Note that
    \[
     M_V = \Hom_{\tK\cdot \tKp}(\CH_{\sigma,\sigma'},Q_V) \simeq \Hom_{\tK}(\sigma,\pi)\otimes \Hom_{\tKp}(\sigma',V),
    \]
    where all these Hom spaces are nonzero and finite-dimensional.
    Now using Lemmas~\ref{Lem: rank eq} and \ref{Lem: theta cyclic}, we obtain
    \[
     \dim \Hom_{\tK}(\sigma,\pi) =  \dim \Hom_{\tKp}(\sigma',V)
    \]
    for both $V=\Theta(\pi)$ and $V=\theta(\pi)$.
\end{proof}

\section{Proof of main theorems}\label{Sec: Proof main}

\subsection{\JH multiplicity of the small theta lift}

We are now ready to prove that the \JH multiplicity of $\theta(\pi)$ in $\Theta(\pi)$ is one.

\begin{thm}\label{Thm: JH small theta}
  For every $\pi$ in $\CR(\g,\tK,\omega)$, we have 
  \[
  [\Theta(\pi) : \theta(\pi)] = 1.
  \]
\end{thm}

\begin{proof}
  Assume $[\Theta(\pi) : \theta(\pi)] = t$, obviously we have $t \geq 1$.
  Take a lowest degree $\tKp$-type $\tau$ in $\theta(\pi)$. 
  Using Theorem~\ref{Thm: LDK multiplicity eq}, we have 
  \[
  \dim \Hom_{\tKp}(\tau,\theta(\pi))
 = \dim \Hom_{\tKp}(\tau,\Theta(\pi))
 \geq t \cdot \dim \Hom_{\tKp}(\tau,\theta(\pi)),
  \]
  where the last inequality follows from the fact that $\tKp$ is compact.
  Thus we have $t=1$
\end{proof}

\subsection{The smooth version}\label{Subsec: smooth vertion}
Recall that we have fixed the smooth Weil representation $\omega^\infty$ in the introduction. Denote by $\CR(\tG,\omega^\infty)$ the set of isomorphism classes of irreducible Casselman--Wallach representations $\Pi$ of $\tG$ such that
\[
    \Hom_{\tG}(\omega^\infty,\Pi)\neq 0,
\]
where $\Hom_{\tG}$ denotes the space of continuous $\tG$-equivariant linear maps. Define $\CR(\tGp,\omega^\infty)$ similarly.
For $\Pi\in\CR(\tG,\omega^\infty)$, consider
\[
    \CN_\Pi
    :=
    \bigcap_{\phi\in\Hom_{\tG}(\omega^\infty,\Pi)}
    \ker\phi,
\]
which is a closed $\tG\cdot\tGp$-invariant subspace of $\omega^\infty$.
We have the following theorem for smooth theta lifting.

\begin{thm}[{\cite[Theorems 1 and 1A]{Howe1989TranscendingClassicalInvTheory}}]
For every $\Pi\in\CR(\tG,\omega^\infty)$, there is an isomorphism
\[
    \omega^\infty/\CN_\Pi
    \simeq
    \Pi\otimeshat\Theta^\infty(\Pi)
\]
of smooth $\tG\cdot\tGp$-representations, where $\Theta^\infty(\Pi)$ is a Casselman--Wallach representation of $\tGp$. In particular, $\Theta^\infty(\Pi)$ has finite length. 

Moreover, $\Theta^\infty(\Pi)$ has a unique irreducible quotient $\theta^\infty(\Pi)$, and the correspondence $\Pi \mapsto \theta^{\infty}(\Pi)$ defines a bijection from $\CR(\tG,\omega^{\infty})$ to $\CR(\tGp, \omega^{\infty})$.
\end{thm}

Here $\otimeshat$ represents the completed projective tensor product. We call $\Theta^\infty(\Pi)$ and $\theta^\infty(\Pi)$ the smooth full and small theta lifts of $\Pi$, respectively.
The following theorem is the smooth analogue of Theorem~\ref{Thm: JH small theta}.

\begin{thm}\label{Thm: smooth JH small theta}
For every $\Pi\in\CR(\tG,\omega^\infty)$, we have
\[
    [\Theta^\infty(\Pi):\theta^\infty(\Pi)]=1.
\]
\end{thm}

\begin{proof}
  Let $\pi$ be the underlying irreducible $(\g,\tK)$-module of $\Pi$. Since $\omega$ is dense in $\omega^\infty$, we have $\pi\in\CR(\g,\tK,\omega)$, and a natural surjective map
  \[
    \Theta(\pi)\twoheadrightarrow \Theta^\infty(\Pi)^{\mathrm{HC}}.
  \]
  Here and henceforth, we use $\Theta^\infty(\Pi)^{\mathrm{HC}}$ and $\theta^\infty(\Pi)^{\mathrm{HC}}$ to denote the underlying Harish--Chandra $(\g',\tKp)$-modules of $\Theta^\infty(\Pi)$ and $\theta^\infty(\Pi)$, respectively.

  By the uniqueness of the irreducible quotient of $\Theta(\pi)$, we have $\theta(\pi) = \theta^{\infty}(\Pi)^{\mathrm{HC}}$. According to Theorem~\ref{Thm: JH small theta}, we obtain
  \[
    1 \leq [\Theta^\infty(\Pi):\theta^\infty(\Pi)] = [\Theta^\infty(\Pi)^{\mathrm{HC}}:\theta^\infty(\Pi)^{\mathrm{HC}}] \leq [\Theta(\pi):\theta(\pi)] =1.
  \]
  Thus, the theorem follows.
\end{proof}

\section{A lowest degree $K$-type with multiplicity two}\label{Sec: mult two example}
In this section, we study an irreducible representation occurring in the theta correspondence which has a lowest degree $K$-type of multiplicity two.

Let $G=\Sp_4(\BR)$ and $G'=\RO(m,m)$, where $m \geq 4$. Then $(G,G')$ is a reductive dual pair in $\Sp_{8m}(\BR)$ in the stable range, with $G$ the smaller member. 
Since the orthogonal space is even-dimensional, theta correspondence in this case can be formulated directly in terms of representations of $G$ and $G'$.
Moreover, by~\cite[Theorem 1]{PP2008stablerange}, every irreducible representation of $G$ occurs in the theta correspondence.

Let $B$ be the Borel subgroup of $G$ consisting of upper triangular matrices, with Levi decomposition $B=AN$, where 
\[
A:= \left\{a(t_1,t_2) := \begin{pmatrix}
	t_1 & & &\\
	& t_2 & &\\
	& & t_2^{-1} &\\
	& & &t_1^{-1}
\end{pmatrix}\;\middle|\;t_1,t_2\in \BR^{\times} \right\} \simeq \BR^{\times}\times \BR^{\times}
\]
is the Levi subgroup consisting of diagonal matrices. Consider the normalized parabolic induction
\[
 \pi:= \Ind_B^G (\chi_1 \boxtimes \chi_2),
\]
where $\chi_i(t)=\sgn(t)|t|^{s_i}$ for $i=1,2$. Assume that
$s_1,s_2\in\sqrt{-1}\BR\setminus\{0\}$ and $s_1\neq\pm s_2$.

\begin{lemt}
	The representation $\pi$ is irreducible.
\end{lemt}
\begin{proof}
	Let $W=N_G(A)/A$ be the Weyl group of $G$, it is generated by simple reflections $r_1,r_2$, which act on $A$ by
	\[
	r_1 \cdot a(t_1,t_2)= a(t_2,t_1), \qquad r_2\cdot a(t_1,t_2)=a(t_1,t_2^{-1}).
	\]
    For $w\in W$, choose a representative $\dot w\in N_G(A)$. For any character $\chi$ of $A$, define
    \[
    \chi^w(a):=\chi(\dot w^{-1}a\dot w),
    \qquad a\in A.
    \]
    Since $s_1\neq \pm s_2$ and $s_1, s_2 \neq 0$, we have
	\[
	W_{\chi_1,\chi_2}:= \{w\in W \mid (\chi_1\boxtimes \chi_2)^w \simeq \chi_1\boxtimes \chi_2\} =\{1\}.
	\] 
	Since $G$ has real rank two and $\chi_1\boxtimes\chi_2$ is unitary, it follows from~\cite[Theorem 3]{KS1972Irreducibility} that $\pi$ is irreducible.
\end{proof}

We now study the lowest degree $K$-types of $\pi$ and their multiplicities. 
Let $\theta(g):= g^{-t}$ be the Cartan involution of $G=\Sp_4(\BR)$ such that $A$ is $\theta$-stable, we have $K= G^{\theta}\simeq \RU(2)$ and 
\[
K \cap A = A^{\theta} \simeq\{\pm 1\}\times \{\pm 1\}.
\]
By the Iwasawa decomposition,
\[
 \pi|_K\simeq \Ind_{K\cap A}^K \left((\chi_1 \boxtimes \chi_2)|_{K\cap A}\right)\simeq \Ind_{\{\pm 1\}\times \{\pm 1\}}^{\RU(2)} (\sgn \boxtimes \sgn).
\]
Thus, using Frobenius reciprocity, for $\sigma\in \CR(K)$, we obtain
\[
\mathop{\Hom}\nolimits_K\left(\sigma, \Ind_{\{\pm 1\}\times \{\pm 1\}}^{\RU(2)} (\sgn \boxtimes \sgn)\right)\simeq \mathop{\Hom}\nolimits_{\{\pm 1\}\times \{\pm 1\}}(\sigma, \sgn\boxtimes\sgn).
\]

We parametrize $\CR(K)$ by highest weights $(\lambda_1,\lambda_2)$, where $\lambda_1,\lambda_2 \in \BZ$ and $\lambda_1\geq\lambda_2$.
By~\cite[Proposition 4]{Paul2005Theta}, the degree of $\sigma=(\lambda_1,\lambda_2)\in \CR(K)$ with respect to the dual pair $(G,G')$ is
\[
\degree \sigma = |\lambda_1|+|\lambda_2|.
\]
A direct computation shows that the lowest degree of $\pi$ is 2, with the corresponding lowest degree $K$-types and their multiplicities listed below:
\[
\begin{array}{c|ccccc}
\text{$K$-type $\sigma$} & (1,1) & (1,-1) & (2,0) &(-1,-1) & (0,-2)\\
\hline
\text{multiplicity} & 1 & 2 & 1 &1 & 1
\end{array}.
\]
In particular, the lowest degree $K$-type $(1,-1)$ occurs with multiplicity two. 
We also mention that, in this case, $(1,1)$ and $(-1,-1)$ are the minimal $K$-types of $\pi$ in Vogan's sense, which always occur with multiplicity one.

\section*{Acknowledgments}
ZG and KW thank Professor Xuhua He for his constant support.
ZG and KW are partially supported by the New Cornerstone Science Foundation through the New Cornerstone Investigator Program awarded to Professor Xuhua He. KW is also supported by the National Natural Science Foundation of China (Project No. 123B1004).

%The authors used ChatGPT for proof exploration and improving the exposition. The authors have verified all mathematical statements, proofs, and references and take full responsibility for the contents of the paper.

The main ideas, questions, and direction of this work were developed by the authors. In this process, ChatGPT contributed an initial argument for the equality of the multiplicities of the lowest degree $K$-types, which was reformulated and rewritten by the authors into its current form. 
The authors have verified all mathematical statements, proofs, and references and take full responsibility for the contents of the paper.

\printbibliography

@article {Howe1989TranscendingClassicalInvTheory,
    AUTHOR = {Howe, Roger},
     TITLE = {Transcending classical invariant theory},
   JOURNAL = {J. Amer. Math. Soc.},
  FJOURNAL = {Journal of the American Mathematical Society},
    VOLUME = {2},
      YEAR = {1989},
    NUMBER = {3},
     PAGES = {535--552},
      ISSN = {0894-0347,1088-6834},
   MRCLASS = {22E45},
  MRNUMBER = {985172},
       DOI = {10.2307/1990942},
       URL = {https://doi.org/10.2307/1990942},
}

@incollection {Adams2007ThetaR,
    AUTHOR = {Adams, Jeffrey},
     TITLE = {The theta correspondence over {$\mathbb R$}},
 BOOKTITLE = {Harmonic analysis, group representations, automorphic forms
              and invariant theory},
    SERIES = {Lect. Notes Ser. Inst. Math. Sci. Natl. Univ. Singap.},
    VOLUME = {12},
     PAGES = {1--39},
 PUBLISHER = {World Sci. Publ., Hackensack, NJ},
      YEAR = {2007},
      ISBN = {978-981-277-078-3},
   MRCLASS = {22E45 (11F27)},
  MRNUMBER = {2401808},
MRREVIEWER = {Hongyu\ L.\ He},
       DOI = {10.1142/9789812770790\_0001},
       URL = {https://doi.org/10.1142/9789812770790_0001},
}

@article {Paul2005Theta,
    AUTHOR = {Paul, Annegret},
     TITLE = {On the {H}owe correspondence for symplectic-orthogonal dual
              pairs},
   JOURNAL = {J. Funct. Anal.},
  FJOURNAL = {Journal of Functional Analysis},
    VOLUME = {228},
      YEAR = {2005},
    NUMBER = {2},
     PAGES = {270--310},
      ISSN = {0022-1236,1096-0783},
   MRCLASS = {20G05},
  MRNUMBER = {2175409},
MRREVIEWER = {Tomasz\ Przebinda},
       DOI = {10.1016/j.jfa.2005.03.015},
       URL = {https://doi.org/10.1016/j.jfa.2005.03.015},
}

@article {PP2008stablerange,
    AUTHOR = {Protsak, V. and Przebinda, T.},
     TITLE = {On the occurrence of admissible representations in the real
              {H}owe correspondence in stable range},
   JOURNAL = {Manuscripta Math.},
  FJOURNAL = {Manuscripta Mathematica},
    VOLUME = {126},
      YEAR = {2008},
    NUMBER = {2},
     PAGES = {135--141},
      ISSN = {0025-2611,1432-1785},
   MRCLASS = {22E45 (22E46)},
  MRNUMBER = {2403182},
MRREVIEWER = {Peter\ E.\ Trapa},
       DOI = {10.1007/s00229-007-0161-8},
       URL = {https://doi.org/10.1007/s00229-007-0161-8},
}

@incollection {KS1972Irreducibility,
    AUTHOR = {Knapp, A. W. and Stein, E. M.},
     TITLE = {Irreducibility theorems for the principal series},
 BOOKTITLE = {Conference on {H}armonic {A}nalysis ({U}niv. {M}aryland,
              {C}ollege {P}ark, {M}d., 1971)},
    SERIES = {Lecture Notes in Math.},
    VOLUME = {Vol. 266},
     PAGES = {197--214},
 PUBLISHER = {Springer, Berlin-New York},
      YEAR = {1972},
   MRCLASS = {22E45},
  MRNUMBER = {422512},
MRREVIEWER = {G.\ I.\ Ol\cprime shanski\u i},
}

\end{document}